\documentclass[a4,12pt]{article}%
\usepackage{amsmath}
\usepackage{tikz}
\usepackage{amsfonts}
\usepackage{amssymb}
\usepackage{graphicx}%
\providecommand{\U}[1]{\protect\rule{.1in}{.1in}}
\usetikzlibrary{arrows}
\newtheorem{theorem}{Theorem}

\newtheorem{comment}[theorem]{Remark}

\newtheorem{definition}[theorem]{Definition}

\newtheorem{lemma}[theorem]{Lemma}

\newcommand{\dN}{{{\bf N}}}

\newcommand{\cl}[1]{\overline{#1}}

\newcommand{\graph}{{\rm{graph}}}

\newcommand{\calF}{\mathcal{F}}

\newcommand{\calH}{\mathcal{H}}

\newenvironment{proof}[1][Proof]{\textbf{#1.} }{\ \rule{0.5em}{0.5em}}
\newcounter{figurecounter}
\begin{document}

\title{Browder's Theorem: from One-Dimensional Parameter Space to General Parameter Space %
\thanks{We thank John Levy for useful comments.
The first author acknowledges the support of the Israel Science Foundation, Grant \#217/17.}}

\author{Eilon Solan%
\thanks{The School of Mathematical Sciences, Tel Aviv
University, Tel Aviv 6997800, Israel. e-mail: eilons@post.tau.ac.il.}
 and Omri N.~Solan%
 \thanks{Einstein Institute of Mathematics, the Hebrew University of Jerusalem, Jerusalem 9190401, Israel. e-mail: omrisola@post.tau.ac.il.}}

\maketitle

\begin{abstract}
A parametric version of Brouwer's Fixed Point Theorem,
which is proven using the fixed-point index,
states
that for every continuous mapping $f : (X \times Y) \to Y$,
where $X$ is nonempty, compact, and connected subset of a Hausdorff topological space
and $Y$ is a nonempty, convex, and compact subset of
a locally-convex topological vector space,
the set of fixed points of $f$, defined by $C_f := \{ (x,y) \in X \times Y \colon f(x,y)=y\}$,
has a connected component whose projection onto the first coordinate is $X$.
In this note we provide an elementary proof for this result, using its reduction to the case $X = [0,1]$.
\end{abstract}

\noindent
Keywords: Browder's Theorem, fixed points, connected component.

\bigskip

\noindent
MSC2010: 54H25.

\section{Introduction}

Brouwer's Fixed Point Theorem (Hadamard (1910) and Brouwer (1911)) states that every continuous mapping of a finite-dimensional simplex into itself
has a fixed point.
This result was later generalized to nonempty, convex, and compact subsets of topological vector spaces, see, e.g.,
Schauder (1930), Tychonoff (1935), and Dyer (1956).

One of the results of Browder (1960) implies the following parametric version of Brouwer's fixed point theorem:
\begin{theorem}[Browder, 1960]
\label{theorem:simple}
Let $Y$ be a nonempty, compact, and convex subset of a locally-convex topological vector space,
and let $f : (X \times Y) \to Y$ be a continuous mapping, where $X = [0,1]$.
Denote by
\[ C_f := \{ (x,y) \in X \times Y \colon f(x,y)=y\} \]
the set of fixed points of $f$.
Then $C_f$ has a connected component whose projection onto the first coordinate is $X$.
\end{theorem}

This result has been used in a variety of topics,
like nonlinear complementarity theory (see, e.g., Eaves, 1971, or Allgower and Georg, 2012),
nonlinear boundary value problems
(Shaw, 1977, Amster, 2021),
the study of global continua of solutions of nonlinear
partial differential equations (see, e.g., Costa and Gon\c{c}alves, 1981, or Massabo and Pejsachowitz, 1984),
theoretical economics (Citanna et al., 2001),
and game theory (see, e.g., Herings and Peeters, 2010, Solan and Solan, 2021,
or Munk and Solan (2021)).

Browder's (1960) proof of Theorem~\ref{theorem:simple} uses the concept of the fixed-point index, and, in fact,
the more general version of Theorem~\ref{theorem:simple} that Browder (1960) proved is stated using the fixed-point index.
Solan and Solan (2022a) extended Theorem~\ref{theorem:simple} to the case that $X$ is any nonemtpy, connected, and compact subset of a Hausdorff topological space,
still using the fixed-point index.
Solan and Solan (2022b) provided an elementary proof of Theorem~\ref{theorem:simple} that does not use the fixed-point index,
but rather a standard separation theorem, Tietze's extension theorem, and Brouwer's fixed point theorem.

In this note we provide an elementary proof to the extension of Solan and Solan (2022a),
and show that also for general parameter spaces,
one does not need to use the fixed-point index to derive Theorem~\ref{theorem:simple}.
Specifically, to prove the extension to general parameter spaces we use the version of the theorem
for one-dimensional parameter spaces coupled with elementary topological arguments.

The paper has several contributions.
First, since our proof uses only elementary topological arguments,
it can be followed by anyone with basic knowledge in topology.
Second, since our proof does not use the fixed-point index, it may be possible to extend it to more general spaces,
where the fixed-point index does not exist or is not known to exist.
And third, we show that the extension of Browder's Theorem for general parameter spaces follows from the version of the theorem
for one-dimensional parameter space.

An interesting problem that remains open
and may be useful for applications is whether Browder's theorem extends to the case where the parameter space is not necessarily compact.
We hope that our proof will pave the road to establishing this extension.

\section{The Statement of the General Browder's Theorem}

The extension of Theorem~\ref{theorem:simple} to general parameter spaces is as follows.

\begin{theorem}[Solan and Solan, 2022a]
\label{theorem:general}
Let $X$ be a nonemtpy, connected, and compact subset of a Hausdorff topological space,
let $Y$ be a nonempty, compact, and convex subset of a locally-convex topological vector space,
and let $f : X \times Y \to Y$ be a continuous mapping.
Denote by
\[ C_f := \{ (x,y) \in X \times Y \colon f(x,y)=y\} \]
the set of fixed points of $f$.
Then $C_f$ has a connected component whose projection onto the first coordinate is $X$.
\end{theorem}

In the next section we will provide an elementary proof of Theorem~\ref{theorem:general}.
We here sketch the main ideas behind the proof.
Suppose first that $X$ possesses a space-filling curve,
that is, there exists a continuous and surjective mapping $\varphi : [0,1] \to X$.
The mapping $g := f \circ (\varphi,\mathrm{Id}_Y) : [0,1] \times Y \to Y$ is a composition of two continuous mappings, hence continuous, and by Theorem~\ref{theorem:simple} the set $C_g$ has a connected component, denoted $\widehat C$, whose projection onto the first coordinate is $X$.
A continuous image of a connected set is connected,
hence
the set $D := \{(\varphi(t),x) \colon (t,x) \in \widehat C\}$ is a connected
subset
of $C_f$
whose projection onto $X$ is $X$.
The projection of the connected component of $C_f$ that contains $D$ onto the first coordinate is $X$.

Suppose now that $X$ does not possess a space-filling curve.
We will define a family of mappings that will approximate in some sense the mapping $g$ defined above.
Let $G$ be a finite open cover of $X$.
For every $O \in G$ choose a point $x_O \in O$.
By possibly replicating some of the points in $(x_O)_{O \in G}$,
we can construct a sequence $(x_0,x_1,\dots,x_{N(G)})$ such that
(a) the set $\{x_0,x_1,\dots,x_{N(G)}\}$ coincides with $(x_O)_{O \in G}$,
(b) for each $i$, the points $x_i$ and $x_{i+1}$ lie respectively in elements $O_i$ and $O_{i+1}$ of $G$ that have a nonempty intersection.
Since $X$ is connected, such a sequence exists.
The sequence $(x_1,x_2,\dots,x_{N(G)})$ should be thought of as an approximation of a space-filling curve.
We next define a continuous mapping $g^G : [0,1] \times Y \to Y$ as follows:
the section $g^G(\frac{i}{N(G)},\cdot)$ coincides with $f(x_i,\cdot)$,
and in each interval $(\frac{i}{N(G)},\frac{i+1}{N(G)})$ the mapping $g^G$ is extended linearly
(recall that $Y$ is a vector space).

The mapping $g^G$ satisfies the conditions of Theorem~\ref{theorem:simple},
hence its set of fixed points has a connected component $\widehat C^G$ whose projection onto the first coordinate is $[0,1]$.
We next define a proper subset of $X \times Y$ that in some sense approximates the set of fixed points of $f$.
Given a finite open cover $F$ of $Y$, define a subset $D^{G,F}$ of $X \times Y$ as a finite union
of sets $\cl{O_X} \times \cl{O_Y}$,
where $O_X \in G$, $O_Y \in F$, and $\cl{O_X}$ and $\cl{O_Y}$ are the closures of these sets;
the set $\cl{O_X} \times \cl{O_Y}$ is taken in this union if and only if
there are $i \in \{1,2,\dots,N(G)\}$, $t \in [0,1]$, and $y \in O_Y$
such that $\{x_i,x_{i+1}\} \cap O_X \neq\emptyset$ and $(\frac{i+t}{N(G)},y) \in \widehat C^G$.
The set $D^{G,F}$ is a union of a finite number of closed sets, and hence closed.

The set of closed subsets of a topological space is compact in the Vietoris topology (also called the finite topology and the exponential topology).
The above construction of the set $D^{G,F}$ can be done for all open covers $G$ of $X$ and $F$ of $Y$.
We will show that any accumulation point of the sequence $(D^{G,F})$ as the covers $G$ and $F$ become finer
(and the cover $F$ of $Y$ consists only of convex sets)
is a connected subset of $C_f$ whose projection onto the first coordinate is $X$.
This will complete the proof of Theorem~\ref{theorem:general}

\section{Proof of Theorem~\ref{theorem:general}}

This section is devoted to the elementary proof of Theorem~\ref{theorem:general}
using its reduction to the one-dimensional case.

\subsection{Finite open covers}

A \emph{finite open cover} of a topological space $Z$
is a finite collection of nonempty open sets whose union is $Z$.
Let $\calF_Z$ be the set of all finite open covers of $Z$.
When $Z$ is a locally-convex space,
a \emph{finite open convex cover} of $Z$
is a finite collection of open convex sets whose union is $Z$.
In this case, denote by $\calF_Z^C$ the set of all finite open convex covers of $Z$.


Define a partial order $\geq$ on $\calF_Z$ as follows:
For every $F_1,F_2 \in \calF_Z$, we say that $F_1 \geq F_2$ if and only if each open set in $F_1$ is a subset of an open set in $F_2$.
For every $F_1,F_2 \in \calF_Z$ there is $F_3 \in \calF_Z$ such that $F_3 \geq F_1$ and $F_3 \geq F_2$,
for example,
\[ F_3 = \{ O_1 \cap O_2 \colon O_1 \in F_1, O_2 \in F_2, O_1 \cap O_2 \neq \emptyset\}. \]
It follows that the pair $(\calF_Z,\geq)$ is a directed set.
For a similar reason, $(\calF_Z^C,\geq)$ is a directed set.

Note that if $\calF_Z$ is a directed set and $F \in \calF_Z$,
then the set $\{F' \in \calF_Z \colon F' \geq F\}$ is a directed set.

\subsection{A one-dimensional version of $f$}

Let $G \in \calF_X$ be a finite open cover of $X$.
For every element $O \subseteq G$ let $x_O \in O$ be arbitrary,
and set
\[ X^G := \{ x_O \colon O \in G\}. \]
Since $G$ is a finite collection of sets, the set $X^G$ is finite.

Define an undirected graph (in the graph-theoretic sense), denoted $\graph(G)$, whose set of vertices is $X^G$
and where there is an edge $x_{O_1} \leftrightarrow x_{O_2}$ if and only if $O_1 \cap O_2 \neq\emptyset$.
Since $X$ is connected, $\graph(G)$ is a connected graph.

Let $x^G_0,x^G_1,\dots,x^G_{N(G)}$ be a finite sequence of points in $X^G$ (some of the points may appear several times in the sequence)
that satisfies the following properties:
\begin{enumerate}
\item   $X^G = \{x^G_0,x^G_1,\dots,x^G_{N(G)}\}$; that is, the sequence $x^G_0,x^G_1,\dots,x^G_{N(G)}$ contains all points in $X^G$.
\item   Every two adjacent elements along the sequence are neighbors in $\graph(G)$,
that is,
the edge $x^G_i \leftrightarrow x^G_{i+1}$ is in $G$ for every $i$, $0 \leq i \leq N(G) -1$.
\end{enumerate}

Define a mapping $g^G : [0,1] \times Y \to Y$ as follows:
\begin{eqnarray*}
g^G(s,y) := \left\{
\begin{array}{lll}
f(x^G_i,y), & \ \ \ & y \in Y, s = \frac{i}{N(G)} \hbox{ for some } i \in \{0,1,\dots, N(G)\},\\
(1-t)f(x^G_i,y) + tf(x^G_{i+1},y), & & y \in Y, s = \frac{i+t}{N(G)} \hbox{ for some } t \in (0,1).
\end{array}
\right.
\end{eqnarray*}
In words, in points $s = \frac{i}{N(G)}$, the section $g^G(s,\cdot)$ coincides with $f(x^G_i,\cdot)$;
and for every $y \in Y$, the section $g^G(\cdot,y)$ is piecewise linear.

The mapping $g^G$ is continuous.
Denote its set of fixed points by
\[ C^{G} := \bigl\{ (s,y) \in [0,1] \times Y \colon g^G(s,y) = y) \bigr\} \subseteq [0,1] \times Y. \]
By Browder's Theorem for a one-dimensional parameter space (see Theorem~\ref{theorem:simple}),
there is a connected component $\widehat C^G$ of $C^{G}$ whose projection onto the first coordinate is $[0,1]$.
Since $g^G$ is continuous, $\widehat C^G$ is closed.

\subsection{An approximation of $\widehat C^G$}

We here define for every finite open convex cover $F$ of $Y$ a set $D^{G,F} \subseteq X \times Y$ that, in some sense,
approximates $\widehat C^G \subseteq [0,1] \times Y$ by a union of rectangles.
For every open set $O$ in a topological space $Z$,
denote by $\cl{O}$ its closure in $Z$.

Let $F \in \calF_Y^C$ be a finite open convex cover of $Y$.
Define
\[ D^{G,F} :=
\bigcup \left\{ \cl{O_X} \times \cl{O_Y} \colon
\begin{array}{lll}
O_X \in G, O_Y \in F,\\
\exists i \in \{0,1,\dots,N(G)-1\}, \exists y \in O_Y, \exists t \in [0,1]\\
\hbox{ such that } \\
x^G_i \in O_X \hbox{ or } x^G_{i+1} \in O_X,
\hbox{ and }
(\frac{i+t}{N(G)},y) \in \widehat C^G
\end{array}
\right\} \subseteq X \times Y. \]

The set $D^{G,F}$ is the union of finitely many closed set, and hence compact.
Let $K(X \times Y)$ be the collection of all closed subsets of $X \times Y$.
By Theorem 4.2 in Michael (1951), this set is compact when endowed with the Vietoris topology, also called the finite topology.

The product set $\calF_X \times \calF_Y^C$ is directed,
and the mapping $D : \calF_X \times \calF_Y^C \to K(X \times Y)$ that assigns to each pair of finite open covers $G \in \calF_X$
and $F \in \calF_Y^C$ the set $D^{G,F} \in K(X \times Y)$
is a net.

Since $K(X \times Y)$ is compact,
there is a subnet $\calH$ of $\calF_X \times \calF_Y^C$ and a compact set $D^* \in K(X \times Y)$
such that $\lim_{(G,F) \in \calH} D^{G,F} = D^*$.
That is, for every open set $O \subseteq X \times Y$ that contains $D^*$ there is a directed subnet $\calH^* \subseteq \calH$
such that $D^{G,F} \subseteq O$ for every $(G,F) \in \calH^*$.

\subsection{Properties of $D^*$}

We will prove Theorem~\ref{theorem:general} by showing the following three properties of $D^*$:
\begin{enumerate}
\item The projection of $D^*$ onto the first coordinate is $X$ (Lemma~\ref{lemma:cover}).
\item $D^*$ is connected (Lemma~\ref{lemma:connected}).
\item $D^* \subseteq C_f$ (Lemma~\ref{lemma:subset}).
\end{enumerate}
These three properties will imply that the connected component of $C_f$ that contains $D^*$ satisfies the conclusion of Theorem~\ref{theorem:general}.
The third property will prove to be the most challenging among the three.

\begin{lemma}
\label{lemma:cover}
The projection of $D^*$ onto the first coordinate is $X$.
\end{lemma}

\begin{proof}
Assume by way of contradiction that there is $x^* \in X$ such that $(\{x^*\} \times Y) \cap D^* = \emptyset$.
Since $X$ is Hausdorff and compact,
there are two disjoint open sets $O_X,O'_X \subseteq X$ such that $x^* \in O_X$ and $D^* \subseteq O'_X \times Y$.
Since $\lim_{(G,F) \in \calH} D^{G,F} = D^*$, it follows that for every $(G,F) \in \calF_X \times \calF_Y^C$
there is $(G',F') \in \calH$ such that $D^{G',F'} \subseteq O'_X \times Y$.
This is a contradition, since for every $G' \in \calF_X$ that contains an element $O_{G'} \subseteq O_X$ there is a point $(x,y) \in D^{G',F'}$ with $x \in O_X$.
\end{proof}

\begin{lemma}
\label{lemma:connected}
$D^*$ is connected.
\end{lemma}

\begin{proof}
Assume by way of contradiction that $D^*$ is not connected.
Then there are two disjoint open sets $O_1$ and $O_2$ in $X \times Y$ such that
(a) $D^* \cap O_1 \neq \emptyset$,
(b) $D^* \cap O_2 \neq \emptyset$,
and
(c) $D^* \subseteq O_1 \cup O_2$.

The set $O_1 \cup O_2$ is open in $X \times Y$ and contains $D^*$.
Hence, there is a directed subset $\calH^*$ of $\calH$ such that $D^{G,F} \subseteq O_1 \cup O_2$ for every $(G,F) \in \calH^*$.
Since $\lim_{(G,F) \in \calH^*} D^{G,F} = D^*$,
it follows that $D^{G,F}$ intersects both $O_1$ and $O_2$, for every $(G,F) \in \calH^*$.

Fix $(G,F) \in \calH^*$.
We will use the following property of the set $D^{G,F}$.
Consider the undirected graph, denoted $\graph(D^{G,F})$, whose vertices are
the sets $(\cl{O_X} \times \cl{O_Y})_{O_X \in G, O_Y \in F}$ that are contained in $D^{G,F}$,
and there is an edge between $\cl{O_X} \times \cl{O_Y}$ and $\cl{O'_X} \times \cl{O'_Y}$
if and only if $\cl{O_X} \times \cl{O_Y}$ and $\cl{O'_X} \times \cl{O'_Y}$ intersect.
Since the set $\widehat C^G$ is connected, $\graph(D^{G,F})$ is connected as well.

Since $\graph(D^{G,F})$ is connected,
it follows that for each $(G,F) \in \calH^*$ there are $O_X \in G$ and $O_Y \in F$ such that
(a) $\cl{O_X} \times \cl{O_Y} \subseteq D^{G,F}$,
(b) $O_1 \cap (\cl{O_X} \times \cl{O_Y}) \neq \emptyset$,
and
(c) $O_2 \cap (\cl{O_X} \times \cl{O_Y}) \neq \emptyset$.
Indeed, all vertices of $\graph(D^{G,F})$ are subsets of $O_1 \cup O_2$,
at least one of them intersects $O_1$,
and at least one of them intersects $O_2$.
If each vertex of $\graph(D^{G,F})$ was a subset of either $O_1$ or $O_2$,
then the union of the vertices in $O_1$ would have been disjoint of the union of the vertices in $O_2$,
and hence the graph would not have been connected.

For each $(G,F) \in \calH^*$,
fix
$(x_1^{G,F},y^{G,F}_1) \in D^{G,F} \cap O_1$
and
$(x_2^{G,F},y^{G,F}_2) \in D^{G,F} \cap O_2$
that belong to the same vertex of $\graph(D^{G,F})$.

Since $X$ and $Y$ are compact,
there is a directed subset $\calH^{**}$ of $\calH^*$ such that the four limits
$x_1^* := \lim_{(G,F) \in \calH^{**}} x_1^{G,F}$,
$x_2^* := \lim_{(G,F) \in \calH^{**}} x_2^{G,F}$,
$y_1^* := \lim_{(G,F) \in \calH^{**}} y_1^{G,F}$,
and
$y_2^* := \lim_{(G,F) \in \calH^{**}} y_2^{G,F}$ exist.
Since $(x_1^{G,F},y^{G,F}_1)$ and $(x_2^{G,F},y^{G,F}_2)$ belong to the same vertex of $\graph(D^{G,F})$, for every $(G,F) \in \calH^{**}$,
we have $x^*_1 = x^*_2$ and $y^*_1 = y^*_2$.
Since $D^* = \lim_{(G,F) \in \calH^{**}} D^{G,F}$,
we have
$(x^*_1,y^*_1) \in D^* \subseteq O_1 \cup O_2$.
Suppose w.l.o.g.~that $(x^*_1,y^*_1) \in O_1$.
Then there exists $(G,F) \in \calH^{**}$ such that
$(x^{G,F}_2,y^{G,F}_2) \in O_1$,
contradicting the fact that
$(x^{G,F}_2,y^{G,F}_2) \in O_2$ and $O_1 \cap O_2 = \emptyset$.
\end{proof}

\bigskip

\begin{lemma}
\label{lemma:subset}
$D^* \subseteq C_f$.
\end{lemma}

To prove Lemma~\ref{lemma:subset} we need one further definition and an auxiliary result.

\begin{definition}
Let $F$ be a finite open cover of the topological space $Z$,
and let $z,z' \in Z$.
We say that $z$ and $z'$ are $(0,F)$-neighbors if these two points lie in the same element of $F$.
For every $k \geq 0$, we say that $z$ and $z'$ are $(k+1,F)$-neighbors if
there is $z'' \in Z$ such that $z$ and $z''$ are $(0,F)$-neighbors and $z''$ and $z'$ are $(k,F)$-neighbors.
\end{definition}

Note that if $z$ and $z'$ are $(k,F)$-neighbors, then in particular they are $(k+1,F)$-neighbors.
Note also that if $z \in O \in F$ and $z' \in \cl{O}$ then $z$ and $z'$ are 1-neighbors.

Since $f$ is continuous and $X$ and $Y$ are compact and Hausdorff, we have the following result.

\begin{lemma}
\label{lemma:important}
Let $x^* \in X$, let $y^* \in Y$,
let $O^1_Y$ be an open convex set in $Y$ that contains $f(x^*,y^*)$,
and let $O^2_Y$ be an open convex set in $Y$ that contains $y^*$.
There are a finite open cover $\widehat G$ of $X$ and a fiinte open convex cover $\widehat F$ of $Y$
such that the following properties hold:
\begin{itemize}
\item[(D.1)]
$f(x,y) \in O^1_Y$
for every $x \in X$ and every $y \in Y$ that satisfy that
$x$ and $x^*$ are $(4,\widehat G)$-neighbors
and $y$ and $y^*$ are $(2,\widehat F)$-neighbors.
\item[(D.2)]  $y \in O^2_Y$ for every $y \in Y$ that is a $(2,\widehat F)$-neighbor of $y^*$.
\end{itemize}
\end{lemma}

We first prove Lemma~\ref{lemma:subset} using Lemma~\ref{lemma:important}.
We will then prove Lemma~\ref{lemma:important}.

\bigskip

\begin{proof}[Proof of Lemma~\ref{lemma:subset}]
Fix $(x^*,y^*) \in D^*$.
Our goal is to show that $(x^*,y^*) \in C_f$,
namely, $f(x^*,y^*) = y^*$.
Since $Y$ is Hausdorff, it is sufficient to prove that $f(x^*,y^*) \in O_Y$
for every open set $O_Y$ in $Y$ that contains $y^*$.
Since $Y$ is locally convex,
it is sufficient to prove this property whenever $O_Y$ is convex.

Fix then an open convex set $O^1_Y$ in $Y$ that contains $y^*$.
Assume by contradiction that $f(x^*,y^*) \not\in O^1_Y$.
Then we can assume w.l.o.g.~that there is an open neighborhood $O^2_Y$ of $f(x^*,y^*)$ that is disjoint of $O^1_Y$.
Indeed, since $Y$ is Hausdorff, 
there are two open sets $\widehat O^1_Y$ and $\widehat O^2_Y$ in $Y$ that contain $y^*$ and $f(x^*,y^*)$, respectively. Consider then $\widehat O^1_Y \cap O^1_Y$ and $\widehat O^2_Y$.

%
Let $\widehat G$ be a finite open cover of $X$
and let $\widehat F$ be a finite open convex cover of $Y$ that satisfy (D.1)--(D.2).
Let $\widehat O_X$ be an element of $\widehat G$ that contains $x^*$,
and let $\widehat O_Y$ be an element of $\widehat F$ that contains $y^*$.

Let $\widehat\calF_X \subseteq \calF_X$ be the directed set of finite open covers $G$ of $\calF_X$ that satisfy $G \geq \widehat G$.

The intersection $D^* \cap (\widehat O_X\times \widehat O_Y)$ is not empty, since it contains $(x^*,y^*)$.
It follows that there exists a directed subset $\widehat \calF^{*}_X$ of $\widehat \calF_X$ such that
$D^{G,\widehat F} \cap (\widehat O_X \times \widehat O_Y) \neq \emptyset$ for every $G \in \widehat \calF^{*}_X$.

For every $G \in \widehat \calF^{*}_X$ fix $(x^G,y^G) \in D^{G,\widehat F} \cap (\widehat O_X\times \widehat O_Y)$.
In particular, $x^G$ and $x^*$ are 0-neighbors, and $y^G$ and $y^*$ are 0-neighbors.
By the definition of $D^{G,\widehat F}$,
there are $x'^G \in O'_X \in G$ and $y'^G \in O'_Y \in \widehat F$ such that
$(x^G,y^G) \in \cl{O'_X} \times \cl{O'_Y} \subseteq D^{G,\widehat F}$.
In particular, $x'^G$ and $x^G$ are $(1,G)$-neighbors, and
$y'_G$ and $y^G$ are $(1,\widehat F)$-neighbors,
and therefore $x'^G$ and $x^*$ are 2-neighbors, and $y'^G$ and $y^*$ are $(2,\widehat F)$-neighbors.

By the definition of $D^{G,\widehat F}$, there are $x^G_1, x^G_2 \in X$ and $t^G \in [0,1]$ such that
\begin{equation}
\label{equ:1}
y'^G = (1-t^G) f(x^G_1,y'^G) + t^G f(x^G_2,y'^G),
\end{equation}
where
(a) w.l.o.g., $x^G_1$ and $x'^G$ are $(0,G)$-neighbors,
and
(b) $x^G_1$ and $x^G_2$ are $(1,G)$-neighbors.
Since $G \geq \widehat G$, the points $x^G_1$ and $x'^G$ are $(0,\widehat G)$-neighbors
and $x^G_2$ and $x'^G$ are $(1,\widehat G)$-neighbors.

It follows that $x^G_1$ and $x^*$ are $(3,\widehat G)$-neighbors
and $x^G_2$ and $x^*$ are $(4,\widehat G)$-neighbors.
By (D.1),
$f(x^G_1,y'^G)$ and $f(x^G_2,y'^G)$ lie in $O^1_Y$.
Since $O^1_Y$ is convex,%
\footnote{This is the only argument in the proof that requires that $Y$ is locally convex.} 
Eq.~\eqref{equ:1} implies that $y'^G \in O^1_Y$.

On the other hand,
Since $y^*$ and $y'^G$ are $(2,\widehat F)$-neighbors,
it follows that $y'^G \in O^2_Y$,
which contradicts the assumption that $O^1_Y$ and $O^2_Y$ are disjoint.
\end{proof}

\bigskip

We complete the proof of Theorem~\ref{theorem:general} by proving Lemma~\ref{lemma:important}.

\bigskip

\begin{proof}[Proof of Lemma~\ref{lemma:important}]
Since $X$ and $Y$ are compact and $f$ is continuous,
and since $f(x^*,y^*)$ lies in the open set $O_Y$,
there are open sets $O^1_X$ in $X$ and $O^1_Y$ in $Y$ such that
\[ f(x,y) \in O_X, \ \ \ \forall x \in O^1_X, \forall y \in O^1_Y. \]
Since $Y$ is locally convex, compact, and Hausdorff,
there is a finite open cover $F \in \calF_Y^C$
such that for every $y \in Y$ that is a $(2,F)$-neighbor of $y^*$ and each
$O'_Y \in F$ that contains $y$ we have $O'_Y \subseteq O^1_Y$.
An analogous statement for $X$ establishes the lemma.
\end{proof}

\begin{comment}
When the spaces $X$ and $Y$ are metric, 
the proof can be simplified: 
in this case the Vietoris topology is metrizable by the Hausdorff distance, 
and therefore there is no need to use nets:
Instead of working with the directed sets of all finite open covers of $X$ and finite open convex covers of $Y$,
one can take for every $k \in \dN$ a finite open  cover $G_k$ of $X$ and a finite open convex cover $F_k$ of $Y$ such that the diameters of all sets in $G_k$ and $F_k$ are smaller than $1/k$,
and define $D^*$ to be a limit of the sets $D^{G_k,F_k}$ as $k$ goes to $\infty$.
\end{comment}

\end{document}